\documentclass[11pt]{article}
\usepackage[english]{babel}
\usepackage{epsfig}
\usepackage{amsmath}
\usepackage{amsthm}
\usepackage{amssymb,color}
\usepackage{mathrsfs}
\usepackage[mathscr]{euscript}

\newtheorem{theorem}{Theorem}[section]
\newtheorem{definition}[theorem]{Definition}

\newtheorem{proposition}[theorem]{Proposition}

\newtheorem{remark}[theorem]{Remark}

\newcommand{\R}{\mathbb{R}}

\title{\bf Slice monogenic theta series}

\author{
Fabrizio Colombo\\
Politecnico di Milano\\
Dipartimento di Matematica\\
Via Bonardi, 9\\
20133 Milano, Italy\\
fabrizio.colombo@polimi.it\and
Rolf S\"{o}ren Krau{\ss}har\\
Chair of Mathematics\\
Erziehungswissenschaftliche Fakult\"at\\
Universit\"at Erfurt\\
Nordh\"auser Str. 63\\
99089 Erfurt, Germany\\
soeren.krausshar@uni-erfurt.de\and
Irene Sabadini\\
Politecnico di Milano\\
Dipartimento di Matematica\\
Via Bonardi, 9\\
20133 Milano, Italy\\
irene.sabadini@polimi.it
}
\begin{document}
\maketitle
\begin{abstract}
In this paper we introduce a generalization of theta series in the context of the slice monogenic function theory in $\mathbb{R}^{n+1}$ where me make use of the so-called $*$-exponential function in a hypercomplex variable. Together with the Eisenstein  and Poincar\'e series that we introduced in a previous paper, the theta series construction in this paper completes the fundament of the basic theory of modular forms in the slice monogenic setting. We introduce a suitable generalized Poisson summation formula in this framework and we apply a  properly adapted Fourier transform. As a direct application we prove a transformation formula for slice monogenic theta series. Then we introduce a family of conjugated theta functions. These are used to construct a slice monogenic generalization of the third power of the Dedekind eta function and of the modular discriminant. We also investigate their transformation behavior. Finally, we show that these theta series are special solutions to a  generalization of the heat equation associated with the slice derivative. We round off by discussing the monogenic case.
\end{abstract}

\noindent {\bf Mathematical Review Classification numbers}: 30G35, 11F04

\section{Introduction}

There are several different ways to generalize classical complex function theory together with its related toolkit for tackling classical applications in the two-dimensional framework to higher dimensional settings. One possibility is offered by complex analysis in several complex variables. Another important line of investigation considers functions that take values in (non-commutative) Clifford algebras.

\par\medskip\par

In classical Clifford analysis one considers null solutions to the higher dimensional generalized Cauchy-Riemann operator,
see for instance \cite{bds,ghs}. The associated functions are often called monogenic functions or hyperholomorphic functions. Their related function theory provides a lot of powerful tools like a Cauchy integral formula to successfully tackle many boundary value problems of harmonic functions in higher dimensional Euclidean spaces.

Additionally, up from the 1990s one also started to intensively consider versions of the Cauchy-Riemann or Dirac operator equipped with the hyperbolic metric \cite{Leutwiler}, and more generally, classes of holomorphic Cliffordian functions which all satisfy a homogeneous or an inhomogeneous Weinstein type equation \cite{cgk}. The latter function classes can also be related to eigensolutions of the Laplace-Beltrami operator.

\par\medskip\par

Apart from these versions of monogenic function theories, more recently, a rapidly growing attention has also been paid to the class of slice hyperholomorphic functions, see for instance
\cite{BOOK_ACS,bookfunctional,BOOK_ENTIRE,BOOKGST},  which has become a counterpart theory to the above mentioned function theories. Slice hyperholomorphic functions offer  different applications to operator theory, in particular to spectral theory for several operators as well as to quaternionic operators. See for example \cite{BOOK_ACS,bookfunctional}.
\par\medskip\par

In this paper we deal with slice monogenic functions, namely with slice hyperholomorphic functions with values in a Clifford algebra, which were introduced in \cite{slicecss}. The quaternionic case has been studied and introduced before, see \cite{GS2}.  For more details, see again the aforementioned books.

\par\medskip\par

Both monogenic  and slice monogenic function theories are natural generalizations of classical complex
function theory but they are quite different from each other.
Possible relations between the two theories
were  developed in the context of the Fueter-Sce mapping theorem,  cf. \cite{fueter1,qian,sce}. The Fueter-Sce mapping allows us to construct monogenic functions starting from holomorphic functions and its inversion \cite{CoSaSo} generates slice monogenic functions from axially monogenic functions.

\par\medskip\par

In our previous paper \cite{CKS2016} we described the invariance behavior of slice monogenic functions under arithmetic subgroups of the Ahlfors-Vahlen group that take axially symmetric domains into each other. We also explained how one can construct slice monogenic Eisenstein  and Poincar\'e series that serve as examples of slice monogenic modular forms on these arithmetic groups.
\par\medskip\par
This also provides a nice analogy to similar constructions in higher dimensional function theories in Clifford algebras, in which one also could successfully introduce monogenic and more in general holomorphic Cliffordian Eisenstein  and Poincar\'e series, cf. \cite{cgk,Kra2004}. These in turn could also be connected to particular Maa{\ss} forms on the Ahlfors-Vahlen group, \cite{cgk,EGM90}.
\par\medskip\par
While in complex analysis of one and several complex variables there also exists the possibility to construct modular forms by theta series and theta functions (see for example \cite{f,Freitag,Krieg85,Terras}), a similar analogue of theta series could not be introduced in the classical Clifford analysis setting so far. A main obstacle consisted in the fact that one was not able to find an appropriate monogenic generalization of the exponential function that on the one hand should be periodic and that additionally should have the property $\exp(z+w)=\exp(z)\exp(w)$ at the same time. In general, monogenic functions do not remain monogenic when forming their product. This is consequence of the non-commutativity.
\par\medskip\par
Now, the great advantage of the slice monogenic function theory consists in the fact that one has a product construction in terms of the so-called $*$-product which elegantly compensates the non-commutativity by a suitable construction that respects slice monogenicity. This additional product structure admits the construction of the so-called $*$-exponential function, cf. \cite{AF,BOOK_ENTIRE}.
\par\medskip\par
In this paper we use this $*$-exponential function to introduce two kinds of generalizations of the theta series in the context of slice monogenic function theory in $\mathbb{R}^{n+1}$. Together with the Eisenstein  and Poincar\'e series that we introduced in \cite{CKS2016}, the theta series constructions in this paper nicely complete the fundament of the basic theory of modular forms in the classical slice monogenic setting. We first show that each of these two series actually converges on an axially symmetric domain that canonically generalizes upper half-plane to the slice monogenic setting in higher dimensions.
\par\medskip\par
We introduce a properly adapted generalized Poisson summation for the slice monogenic framework. To this end, we consider an intrinsic  Fourier transform.

\par\medskip\par
 As a direct application we are in position to prove a transformation formula for slice monogenic theta series relating the theta series at a point $x \in \mathbb{R}^{n+1}$ with its value at the inverted point $ \pm x^{-1}$. Additionally to their invariance under the inversion (up to a scaling factor) these series also exhibit a periodicity (either radial or translation periodicity) in the paravector variable. In this sense, the slice monogenic theta series are quasi-modular forms with respect to these transformations.
 \par\medskip\par

Furthermore, we introduce a family of conjugated theta functions and study their transformation behavior. These functions then in turn serve as building blocks to construct further examples of slice monogenic quasi-modular forms in terms of the star product of slice monogenic functions. In particular, we use them to introduce a slice monogenic generalization of the modular discriminant. This provides a key ingredient for further research in the development of the basic theory of automorphic forms in the slice monogenic context.

\par\medskip\par
	
Next  we also show that these theta series are solutions to a generalization of the heat equation associated to the slice derivative, hence providing us also with an application to partial differential equations.

\par\medskip\par

Finally, we use the Fueter-Sce theorem to introduce the monogenic generalization of the theta series and prove a transformation formula in the quaternionic setting. The transformation behavior however is much more complicated than in the slice monogenic setting, involving a sum of several terms and derivatives.

\par\medskip\par
As a future perspective we hope that the new constructions given in this paper will enable us to tackle a series of number theoretical problems arising recently in the context of generalized theta series, functions and integrals. In particular, harmonic theta series and applications to generalized error functions currently represent an important topic of interest in the number theory community, see for example \cite{ABMP,DSMS,FunkeKudla}, just to mention a few of an impressively large amount of papers that have appeared over the last years in this direction. In this sense we hope that the toolkit of slice monogenic function theory could also provide us with some input for the further development in the future.

\section{Preliminaries}
In this section we introduce some preliminary results on
M\"obius transformations in $\R^{n+1}$ and recall the related analyticity concepts within classes of Clifford algebra valued monogenic and slice monogenic functions.

\subsection{Basics on Clifford algebras and notations}

A basis for the real Clifford algebra $\mathbb{R}_{n}$, considered as a vector space, is given by the element $e_{\emptyset}=1$, the canonical basis elements $e_1,e_2,\ldots, e_n$ which satisfy $e_ie_j+e_je_i=-2\delta_{ij}$, as well as all their possible products $e_1,e_2,\ldots, e_n,e_1 e_2,\ldots , e_1 e_n,\ldots , e_{n-1} e_n,\ldots, e_1 e_2 \cdots e_n$. In compact form, the set containing the products is described by $\{e_A \mid  A \subseteq\{1,\ldots,n\}\}$ where $e_{\emptyset} =1$. Thus, an arbitrary element of $\mathbb{R}_n$ has the form $a = \sum\limits_{A \subseteq\{1,\ldots,n\}} a_A e_A$  with real components $a_A$. Here we have set $e_A := e_{l_1}\cdots e_{l_r}$ where $A=(l_1,\ldots ,l_r)$ is a multi-index and the integers $l_1,\ldots ,l_r$ satisfy $1 \le l_1 < \cdots < l_r \le n $.
Next we introduce the Clifford conjugation by $\overline{a} := \sum\limits_A a_A \overline{e_A}$ where
$$
\overline{e_A} = \overline{e_{l_r}} \cdots \overline{e_{l_1}},\; \overline{e_j}=-e_j,\;\;j=0,\ldots,n,\;\;\overline{e_{\emptyset}}=e_{\emptyset}=1.
$$
Furthermore, the Clifford reversion is defined by
$\tilde{a} := \sum\limits_A a_A \tilde{e_A}$ where
$$
\widetilde{e_A} = e_{l_r} \cdots e_{l_1},\; \tilde{e_j}= e_j,\;\;j=1,\ldots,n,\;\;\tilde{e_{\emptyset}}=e_{\emptyset}=1.
$$
We also have $\tilde{a} = \sum_A (-1)^{|A|(|A|-1)/2} a_A e_A$.
Furthermore, we consider the main involution defined by
$$
{e_A}' = {e_{l_1}}' \cdots {e_{l_r}}',\; {e_j}'=-e_j,\;\;j=1,\ldots,n,\;\;{e_{\emptyset}}'=e_{\emptyset}=1.
$$
One has the relation $\overline{a} = \tilde{a'} = \tilde{a}'$.

We will identify the set of paravectors, i.e. elements of the form $x_0+x_1e_1+\cdots +x_ne_n$ with elements in the Euclidean space $\mathbb{R}^{n+1}$ by the isomorphism $x_0+x_1 e_1+\cdots +x_n e_n \mapsto (x_0,x_1,\ldots, x_n)$.
We use the set
$$
\mathbb{S}^{n-1}=\{\omega=a_1 e_1 + \cdots + a_n e_n\ : \ a_1^2+\ldots+a_n^2=1 \}
$$
which can be identified with a sphere in the reduced vector space $\mathbb R^n$ and whose elements $\omega$ all satisfy $\omega^2=-1$. In the complex case addressed by $n=1$ this set simply reduces to the discrete set $\{e_1,-e_1\}$. As soon as $n >1$ this set gets a connected sphere.
\\
The norm $\| x \|$ of a paravector $x$  is  $\|x\| = \left(\sum\limits_{i=0}^n x_i^2\right)^{1/2}$ namely the usual Euclidean norm. This norm can be extended to a pseudo-norm on the whole Clifford algebra by defining $\|a\| := \sqrt{\sum_{A} |a_A|^2}$.
Each non-zero paravector is invertible with inverse $x^{-1} = \frac{\overline{x}}{\|x\|^2}$.

\subsection{M\"obius transformations in $\R^{n+1}$}

As it is broadly well-known, in dimension $n \ge 3$ the set of conformal maps coincides with that of M\"obius transformations. Using Clifford algebras, M\"obius transformations can be written very elegantly in terms of the action of $(2\times 2)$ Clifford algebra valued matrices $\left(\begin{array}{cc} a & b \\ c & d \end{array}\right)$ whose coefficients satisfy special conditions which will be listed below.
The associated group is the general Ahlfors-Vahlen group, cf. \cite{Av,EGM87}.

 \begin{definition}\label{AVcond} The group $GAV(\mathbb{R}^{n+1})$ is the set of matrices $\left(\begin{array}{cc} a & b \\ c & d \end{array}\right)$
equipped with the product of matrices, whose coefficients $a,b,c,d\in\R_{n}$ satisfy the so-called Ahlfors-Vahlen conditions:
\begin{enumerate}
\item[(i)] \ $a,b,c,d$ are products of paravectors from $\mathbb{R}^{n+1}$ (including $0$);
\item[(ii)] \ $a\tilde{d} - b\tilde{c} \in \R \setminus \{ 0 \} $;
\item[(iii)] \ $a\tilde{b},c\tilde{d}  \in \mathbb{R}^{n+1}$ .
\end{enumerate}
\end{definition}
Following for example \cite{EGM87},
M\"obius transformations are defined as action of $GAV(\mathbb{R}^{n+1})$ on $ \R^{n+1}$ by
$$
\left( \left(\begin{array}{cc} a & b \\ c & d \end{array}\right),x\right) \mapsto M\langle x\rangle = (ax+b)(cx+d)^{-1}\in \R^{n+1}.
$$
In the case where $a,b,c,d$ are products of vectors from $\mathbb{R}^n$ the associated group $GAV( \mathbb{R}^n)$ acts transitively on right half-space $x_0>0$, or, respectively, the group $GAV(\mathbb{R} \oplus \mathbb{R}^{n-1})$ acts transitively on upper half-space $x_n > 0$.

Following $\cite{EGM87}$ and others, the whole group $GAV(\mathbb{R}^{n+1})$ can be generated by  four different types of matrices each inducing particularly elementary translations, the Kelvin inversion, rotations and dilations. For details we also refer the interested reader to our recent paper \cite{CKS2016} which treats particular applications to the slice monogenic framework.

\subsection{Two classes of hypercomplex functions}

In this subsection we briefly recall two different basic concepts that generalize holomorphic function theory to higher dimensional real vector spaces.  Concretely speaking, we look at Clifford algebra valued monogenic functions and at Clifford algebra valued slice monogenic functions; the latter function class stands in the main focus of this paper. We briefly explain the connections between these two function classes as well as some of their important properties concerning this paper. In particular, we recall Fueter's theorem that provides us with a key link between holomorphic and slice monogenic functions including a constructive method to obtain slice monogenic functions from holomorphic ones.  We start by recalling the definition of monogenic functions, cf. for instance \cite{bds,ghs}:

\par\medskip\par

{\bf Monogenic functions.} Let $U \subseteq \mathbb{R}^{n+1}$ be an open set. Then a real differentiable function $f: U \rightarrow \mathbb{R}_{n+1}$ that  satisfies $Df = 0$ (respectively $fD = 0$), where $D := \frac{\partial }{\partial x_0}   + e_1 \frac{\partial }{\partial x_1} + \cdots + e_n\frac{\partial }{\partial x_n}$ is the generalized Cauchy-Riemann operator, is called left monogenic (respectively right monogenic), see \cite{bds,ghs}. Due to the non-commutativity of $\mathbb{R}_{n+1}$ for $n>1$, the two
classes of functions do not coincide. However $f$ is left monogenic if and only if $\tilde{f}$ is right monogenic. The generalized Cauchy-Riemann operator factorizes the Euclidean Laplacian $\Delta = \sum_{j=0}^n \frac{\partial^2}{\partial x_j^2}$, since $D\bar D=\bar D D = \Delta$. Every real component of a monogenic function hence is harmonic.

An important property of the $D$-operator is its quasi-invariance under M\"obius transformations acting on the complete Euclidean space $\mathbb{R}^{n+1}$.
\begin{theorem}(cf. {\rm \cite{r85}}). Let $M \in GAV(\mathbb{R}^{n+1})$ and let $f$ be a left monogenic function in the variable
$y=M \langle x \rangle=(ax+b)(cx+d)^{-1}$. Then
\begin{equation}\label{moninv}
g(x):=\frac{\widetilde{cx+d}}{\|cx+d\|^{n+1}}f(M \langle x \rangle)
\end{equation}
 is left  monogenic in the variable $x$ for any $M \in GAV(\mathbb{R}^{n+1})$.
\end{theorem}
Notice that the transformation~(\ref{moninv}) is up to a constant the most general transformation that sends a monogenic function again to a monogenic one by applying a M\"obius transformation in the argument. It requires the particular exponent $n+1$ in the expression of the denominator.
\newpage
{\bf Slice monogenic functions}.
\par\medskip\par
As we mentioned in the Introduction, the class of slice monogenic functions is also widely studied nowadays, see for example the aforementioned books and the references therein for more details.

\begin{definition}\label{slicemon2}
Let $U$ be an open set in $\mathbb{R}^{n+1}$, $f:U\to\mathbb{R}_{n}$.
Let $\omega\in \mathbb{S}^{n-1}$ and let $f_\omega$ be the restriction of $f$ to the
complex plane $\mathbb C_\omega=\{u+\omega v, \ |\ u,v\in\mathbb R \}$.
We say that $f$ is a (left) slice monogenic function if for every
$\omega\in \mathbb{S}^{n-1}$
\begin{equation}\label{def1}
D_{\omega} f_\omega (u +\omega v):=\frac{1}{2}\left(\frac{\partial }{\partial u}+\omega\frac{\partial
}{\partial v}\right)f_\omega (u +\omega v)=0,
\end{equation}
for $u+\omega v\in U$. The set of slice monogenic functions on $U$ is denoted by $\mathcal{SM}(U)$.
\end{definition}
Slice monogenic functions such that $f:\ U\cap\mathbb{C}_\omega\to \mathbb{C}_\omega$ for all $I\in\mathbb{C}_\omega$ are called {\em intrinsic}.
For our purposes, it is convenient to put restrictions on the open sets $U$ that may be considered, namely we shall consider {\em axially symmetric} sets. Let $\omega_0\in \mathbb{S}^{n-1}$. $U$ is axially symmetric if $u +\omega_0 v\in U$ implies that $u +\omega v\in U$ holds for all $\omega\in U$. Moreover, a domain $U$ is called a {\em slice domain}, if $U\cap \mathbb C_\omega$ is connected for all $\omega\in U$.

As it is well known, on axially symmetric slice domains a function is slice monogenic in the standard sense if and only if it of the form $f(u +\omega v)= \alpha(u,v)+\omega \beta(u,v)$, cf. \cite{slicecss}.
\\
So, we consider the following adapted definition, see \cite{GP}:
\begin{definition}
Let $U\subseteq  \mathbb{R}^{n+1}$  be an axially symmetric domain, let $D\subseteq\mathbb{R}^2$ be an open set such that $u +\omega v\in U$ whenever $(u,v)\in D$ and let $f : U \to \mathbb{R}_{n}$.
The function $f$ is a slice function if
there exist two functions $\alpha, \beta: D\subseteq\mathbb{R}^2 \to \mathbb{R}_{n}$ satisfying the following even-odd conditions
$\alpha(u,v)=\alpha(u,-v)$, $\beta(u,v)=-\beta(u,-v)$
such that
\begin{equation}\label{alphabeta}
f(u +\omega v)= \alpha(u,v)+\omega\beta(u,v).
\end{equation}
If, in addition, the functions $\alpha$ and $\beta$ are differentiable and satisfy the Cauchy-Riemann system
 \begin{equation}\label{CRSIST_qua}
\left\{
\begin{array}{c}
\partial_u \alpha-\partial_v\beta=0\\
\partial_u \beta+\partial_v\alpha=0\\
\end{array}
\right.
\end{equation}
the function $f$ is called slice monogenic.
The class of slice monogenic functions defined on $U$ will be denoted by $\mathcal{SM}(U)$.
\end{definition}
We note that slice monogenic intrinsic functions are characterized by the condition that $\alpha$ and $\beta$ are real-valued functions.\\
More generally, let $U$ be an axially symmetric open set. Furthermore, let $f:\ U\to\mathbb{R}_{n}$ be
 a function of the form $f(u+\omega v)= \alpha(u,v)+\omega\beta(u,v)$ with $\alpha(u,v)=\alpha(u,-v)$, $\beta(u,v)=-\beta(u,-v)$.  We say that the slice function $f$ belongs to the class $\mathcal C^k$ on $U$ if $\alpha, \beta$ belong to the  class $\mathcal C^k$ on $D$.
\par\medskip\par

As in the monogenic case, the pointwise multiplication of two slice monogenic functions does not give a slice monogenic function in general. However, in the slice monogenic context it is possible to define a suitable product, called the $*$-product, which is an inner operation on the set of slice monogenic functions. It is defined as follows. Let  $U\subseteq \mathbb{R}^{n+1}$ be an axially symmetric set and let $f,g\in\mathcal{SM}(U)$ with $f(x)=f(u+\omega v)= \alpha(u,v)+\omega\beta(u,v)$,
$g(z)=g(u +\omega v)=  \gamma(u,v)+\omega\delta(u,v)$. Then one defines, see \cite{bookfunctional, GP}
\begin{equation}\label{*prod}
(f*g)(x)=(f*g)(u +\omega v)= (\alpha(u,v)\gamma(u,v) -\beta(u,v)\delta(u,v))+\omega(\beta(u,v)\gamma(u,v)+\alpha(u,v)\delta(u,v)).
\end{equation}
This multiplication coincides with the standard notion of multiplication of two polynomials or of two converging power series in a non-commutative ring, see \cite{fliess}. Specifically, if $f(x)=\sum_{k\geq 0} x^k a_k$ and
$g(x)=\sum_{k\geq 0} x^k b_k$, then
$$
(f*g)(x)=\sum_{n\geq 0}x^n \left(\sum_{k=0}^n a_k b_{n-k}\right).
$$
It is also possible to define an inverse with respect to the $*$-product.
For further information on slice monogenic functions we refer the reader to \cite{bookfunctional}.
We note that the definition in \eqref{*prod} works more in general for slice functions (see \cite{GP}).

\begin{remark}{\rm As we discussed in the Introduction, the class of slice monogenic functions and the class of monogenic functions can be related.

Let $U$ be an axially symmetric open set in $\mathbb R^{n+1}$ and let $f$ be slice monogenic in $U$. By the Fueter-Sce mapping theorem, the function $\Delta^{n-1/2}f$ is monogenic, see \cite{CoSaSo}. To be more precise, $\Delta^{n-1/2}f$ is axially monogenic. Given an axially monogenic function $\breve f$, it makes sense to ask whether it is possible to construct a so-called Fueter primitive, that is a slice monogenic function $f$ such that $\Delta^{n-1/2}f=\breve f$. The answer is positive and the construction of the Fueter primitive is given in \cite{CoSaSo}. This result can be further generalized in order to include the general case of monogenic functions.
}
\end{remark}
\par\medskip\par

\subsection{M\"obius transformations preserving axial symmetry}

In this section we briefly recall which concrete subgroup of M\"obius transformations leaves the axial symmetry property of a set invariant. The direct analogue of the general Ahlfors-Vahlen group in this particular context is the set stabilizer of the $x_0$-axis. The latter is generated by the inversion, dilations, translations in the $x_0$-direction only, and by modified rotations. From  \cite{CKS2016} we recall:
\begin{definition}
The group $GRAV(\mathbb{R}^{n+1})$ is defined by
\begin{equation}\label{GRAVdef}
GRAV(\mathbb{R}^{n+1}) := \Bigg\langle  \left(\begin{array}{cc} 1 & b  \\ 0 & 1 \end{array}\right),
\left(\begin{array}{cc} a & 0 \\ 0 & {{a}}^{-1} \end{array}\right),
\left(\begin{array}{cc} 0 & 1 \\ -1 & 0 \end{array}\right),
\left(\begin{array}{cc} \lambda & 0 \\ 0 & \lambda^{-1}  \end{array}\right)
 \Bigg\rangle.
\end{equation}
where $b \in \R$, $a \in \mathbb{S}^{n-1}$ and $\lambda  \in \R \backslash \{0\}$.
\end{definition}
\par\medskip\par
\begin{proposition} (See {\rm \cite{CKS2016}, Prop. 2.13})
The elements in the group  $GRAV(\mathbb{R}^{n+1})$ take axially symmetric sets into axially symmetric sets.
\end{proposition}

\begin{remark}{\rm Notice that the other transformations, for example rotations not preserving the real axis, are clearly not preserving the axial symmetry of a set.}
\end{remark}

In this context the natural analogue of the special Ahlfors-Vahlen group, consisting of the sense-preserving matrices, is the group
$$
SRAV(\mathbb{R}^{n+1}) := \{M \in GRAV(\mathbb{R}^{n+1}) \mid \det(M) =1 \}
$$
which is generated only by the first three types of matrices listed in \eqref{GRAVdef}. Dilations are not needed.

\begin{remark}\label{remark2.14}
{\rm
A crucial question for the topic of our paper is to understand what are the appropriate generalizations of upper half-plane in the axial symmetric setting. Let us write a paravector $x=x_0+e_1x_1+\ldots+e_nx_n$ from $\mathbb{R}^{n+1}\setminus\mathbb{R}$ in the form $x = x_0 + \omega r$ with $\omega = \frac{\underline{x}}{\|\underline{x}\|}$, $x_0 \in \mathbb{R}$, $r > 0$.

One possible generalization of complex upper half-plane in the slice monogenic setting is the set
$$
H:= \bigcup_{\omega \in \mathbb{S}^{n-1}} \mathbb{C}^+_{\omega} = \mathbb{R}^{n+1} \backslash \mathbb{R}.
$$
Here, by $\mathbb{C}^+_{\omega}$ we mean the complex upper half-plane associated to the imaginary unit $\omega$ and where the $x_0$-axis is excluded. By construction, the groups $GRAV(\mathbb{R}^{n+1})$ and  $SRAV(\mathbb{R}^{n+1})$ leave $H$ invariant. In particular, this set is invariant under the usual translations $x \mapsto x+b$, $b\in\mathbb R$.

We note that while working in $H$ given a function of the form \eqref{alphabeta}, there is no need to impose the even-odd conditions on $\alpha$ and $\beta$, since $v>0$.

Another axially symmetric domain that can be considered is the right half-space
$$
H^{r} := \{x \in\mathbb{R}^{n+1} \  \mid x_0 > 0\},
$$
which can also be seen as
$H^{r}= \bigcup_{\omega \in \mathbb{S}^{n-1}} \mathbb{C}^r_{\omega}=\{x = x_0 + \omega y \mid x_0 > 0,\ y\in\mathbb R\}$
where $\mathbb{C}^r_{\omega}=\{z = x_0 + \omega y \in\mathbb{C}_\omega\ | \ x_0>0 \}$.
This half-space is also axially symmetric with respect to the $x_0$-axis. But note that translations of the form $x \mapsto x + b$ with $b \in \mathbb{R}$ do not leave $H^{r}$ invariant.

}
\end{remark}
Next we want to understand what are the analogues of the ordinary translations in the $H^{r}$ setting.

\par\medskip\par

An important axial symmetric transformation that leaves the set $H^{r}$ invariant is radial periodicity in the reduced $1$-vector variable $\underline{x}:=x_1 e_1 + \cdots + x_n e_n$. As we already pointed out, any element $x \in \mathbb{R}^{n+1}$ can be written in polar form $x = x_0 + r \omega$ where $r>0$ and where $\omega:= \frac{x_1 e_1 + \cdots + x_n e_n}{(x_1^2 + \cdots + x_n^2)^{1/2}}$. This representation is not unique when $x\in\R$ since $x=x_0+0 \cdot \omega$ for any $\omega\in\mathbb S^{n-1}$.

For any $x\in H$ we shall also write $x=x_0+\omega r$, thus identifying it with $x_0,\omega, r$ i.e. identifying $H$ with $\mathbb{R}\times \mathbb{S}^{n-1}\times \mathbb{R}^+$.

\begin{remark}
Motivated by \cite{IA} and other papers, in $H$ we can consider the notion of radial periodicity also in the slice monogenic setting.
A function $f$ that is slice monogenic in the variable $x=x_0 + x_1 e_1 +\cdots  + x_n e_n=x_0 + r \omega$, $x\in H$, $r>0$, is called {\em radial periodic} in its vector part with period $T>0$ if it satisfies $f(x_0 + r \omega) = f(x_0 + (r+T)\omega)$ for all $r > 0$.
\end{remark}
In the complex case radial periodicity in the reduced variable (which is the imaginary variable) is nothing else than ordinary one-fold periodicity in the imaginary variable. Actually in the complex case  the consideration of the right half-plane $z=x+iy$ with $x>0$ and upper half-plane $z=x+iy$ with $y>0$ is identical. In higher dimensions the geometry is different. In $\mathbb{C}$ ($n=1$) the set $S^{n-1}$ reduced to two isolated points $i$ and $-i$. If $n > 1$ then $\mathbb{S}^{n-1}$ is a connected set. Therefore, $H$ and $H^{r}$ are essentially different sets.

\section{Slice monogenic exponentials}

A crucial aim of this paper is to introduce an appropriate generalization of the famous theta series from the classical complex analysis setting to the slice monogenic setting in a real vector space of general dimension $n+1$, attached to a general $n+1$-dimensional lattice.

To this end we will also be in need of the definition of a suitable exponential function and to define the composition of the exponential function with some slice monogenic functions $f$. To proceed in this direction let $f(x)$ be a function that is slice monogenic on the whole $\R^{n+1}$, namely an entire slice monogenic function. Following \cite{BOOK_ENTIRE}, which is based on \cite{GGS}, one can consider, at least formally, the series
\begin{equation}\label{expseries}
\sum_{k\ge 0}\frac{1}{k!}(f(x))^{*k}.
\end{equation}
We have:
\begin{proposition}
The series \eqref{expseries} converges uniformly over the compact sets of $\mathbb{R}^{n+1}$ and defines a slice monogenic function.
\end{proposition}
\begin{proof}
The proof of the first statement is immediate, since for any fixed compact set $C$ in $\mathbb{R}^{n+1}$ and setting $M_{C}=\max_{C} |f(x)|$, we have $$
\left|\sum_{k\ge 0}\frac{1}{k!}(f(x))^{*k}\right|\leq \sum_{k\geq 0}\dfrac{(2^{n+1})^kM_{C}^k}{k!}<\infty.
$$
To prove the second part, we use the fact that $f(x)$ is a slice monogenic function and, by its definition, so is $(f(x))^{*k}$, i.e. $(f(u+\omega v)))^{*k}=\alpha_k(u,v)+\omega \beta_k(u,v)$, with $(\alpha_k,\beta_k)$ which form an even-odd pair satisfying the Cauchy-Riemann system.
By fixing a basis of the Clifford algebra $\mathbb{R}_{n+1}$ we can write $\alpha_k=\sum_{|A|=0}^{n+1} \alpha_{k,A}e_A$, $\beta_k=\sum_{|A|=0}^{n+1} \beta_{k,A}e_A$, with the pairs $(\alpha_{k,A},\beta_{k,A})$ satisfying the Cauchy-Riemann system, namely $\alpha_{k,A}+\omega\beta_{k,A}$ are holomorphic for all multi-indices $A$. We deduce that
$$
\sum_{k\geq 0} \frac{1}{k!}(f(x))^{*k}= \sum_{k\geq 0} \frac{1}{k!}\sum_{|A|=0}^{n+1} (\alpha_{k,A}+\omega\beta_{k,A})e_A
$$
converges to a function satisfying the Cauchy-Riemann system and so it is slice monogenic (see Definition \ref{slicemon2}).
\end{proof}
We now recall the definition of the $*$-exponential (see \cite{GGS} and also \cite{AF,BOOK_ENTIRE}):
\begin{definition}\label{expdef} Let $U\subseteq\mathbb{R}^{n+1}$ be an axially symmetric set and let $f:\, U \to\mathbb{R}_n$ be a slice monogenic function.
We set
\begin{equation}\label{exp*}
\exp_*(f(x))=\sum\limits_{k\ge 0}\frac{1}{k!}(f(x))^{*k},
\end{equation}
and we call this function $*$-exponential.
\end{definition}
In particular, when $f(x)=a+xb$, $a,b\in{\mathbb{R}}_{n+1}$, we have
\begin{equation}
\exp_{*}(a+xb)=\sum\limits_{k \ge 0} \frac{1}{k !} (a+xb)^{*k}= \sum\limits_{k \ge 0} \frac{1}{k !} \sum_{k=0}^m  
\binom{m}{k}  x^{m-k} a^k b^{m-k}.\nonumber
\end{equation}
\begin{remark}{\rm
Definition \ref{expdef} implies that when $f$ is an intrinsic function, then $\exp_*(f)=\exp(f)$ the classical exponential function.
}
\end{remark}
In the quaternionic case, the properties of the $*$-exponential function have been studied in \cite{AF}, where the authors carefully discuss, in particular, in which cases the equality $\exp_*(f+g)=\exp_*(f)*\exp_*(g)$ holds, see Theorem 4.14 in \cite{AF}. We follow the lines of the proof of that theorem to prove the result below which is enough for our purposes.
\begin{theorem}\label{theoremcomm}
Let $f,g\in\mathcal{SM}(\mathbb{R}^{n+1})$ be commuting with respect to the $*$-product. Then
\begin{equation}
\label{equalityexp}
\exp_*(f+g)=\exp_*(f)*\exp_*(g).
\end{equation}
\end{theorem}
\begin{proof}
By definition and using the fact that $f*g=g*f$ we have:
\begin{eqnarray*}
\exp_*(f+g)&=&\sum_{n=0}^\infty \frac{1}{n!}(f+g)^{*n}\\
&=& \sum_{n=0}^\infty \frac{1}{n!}\sum_{k\leq n} \binom{n}{k} f^{*k}*g^{*(n-k)}\\
&=& \sum_{n=0}^\infty \sum_{k\leq n} \frac{1}{k!(n-k)!} f^{*k}*g^{*(n-k)}\\
&=& \sum_{k=0}^\infty \sum_{n\geq k} \frac{1}{k!} f^{*k}*\frac{1}{(n-k)!}g^{*(n-k)}\\
&=& \sum_{k=0}^\infty  \frac{1}{k!} f^{*k}*\sum_{m=0}^\infty\frac{1}{m!}g^{*m}\\
\end{eqnarray*}
and the statement follows.
\end{proof}

\section{Slice monogenic theta series and their transformation formula}
\subsection{Definition and convergence}

 In the  sequel, let $L \subset \mathbb{R}^{n+1}$ be an arbitrary $(n+1)$-dimensional lattice from $\mathbb{R}^{n+1}$, namely
$$
L = \{q=m_0 \mathfrak{Q}_0 + \cdots + m_n \mathfrak{Q}_n \mid m_0,\ldots,m_n \in \mathbb{Z}\}
$$
where $\mathfrak{Q}_0,\ldots,\mathfrak{Q}_n$ are $\mathbb{R}$-linearly independent elements in $\mathbb{R}^{n+1}$ and $|q|^2\in\mathbb N_0=\mathbb N\cup\{0\}$, for all $q\in L$. A general element $q$ in the lattice $L$ can be written with respect to the canonical basis of $\mathbb{R}^{n+1}$ as $q=\sum_{i=0}^n e_i q_i$. As it is very well-known, see for instance \cite{Conway}, the determinant of the lattice $L$ is defined as the determinant of the Gram matrix $(L)_{lm}$ built by the Euclidean $\mathbb{R}^{n+1}$-inner products of the lattice generators, i.e. $\sum\limits_{i=0}^n {\mathfrak{Q}_l}_i {\mathfrak{Q}_m}_i$. It will be denoted by $\det(L)$.

The dual lattice  also called reciprocal lattice of $L$ will be denoted by  $L^{\sharp}$ and is explicitly defined by
$$
L^{\sharp} = \{y \in \mathbb{R}^{n+1} \mid x_0 y_0 + \cdots + x_n y_n \in \mathbb{Z}\;\;{\rm for\;all}\; x = x_0 + \sum\limits_{i=1}^n x_i e_i \in L\}.
$$
Every lattice satisfies $\det(L^{\sharp}) = \frac{1}{\det(L)}$.
Note that a lattice $L$ is integral if and only if $L \subseteq L^{\sharp}$. A lattice is called unimodular if $L^{\sharp} = L$ which is equivalent to $|\det(L)|=1$.

The slice monogenic $*$-exponential function is appropriate to serve as building blocks for the construction of slice monogenic theta series and a properly on $\mathbb{R}^{n+1}$ intrinsically defined Fourier transform serves as fundamental tool to establish their functional equation. We may introduce theta series in two ways, one associated with $H$ and the another one associated with $H^{r}$. We start by discussing the model $H$. In this framework we introduce:

\begin{definition} (Slice monogenic theta series associated with $H$)\\
Let $x \in H = \{x = x_0 + x_1 e_1 + \cdots + x_n e_n =x_0+\underline{x} \mid \underline{x} \not= 0\}$ written as
$x =x_0 + r \omega$ with $r > 0$ and $\omega = \frac{\underline{x}}{||\underline{x}||}$.
 Let $w\in\mathbb{C}_{\omega}^{n+1}$ be of the form $w=\sum_{i=0}^n e_i(u_i + v_i \omega)=\sum_{i=0}^n e_iw_i$, $u_i$, $v_i\in\mathbb{R}$, with $w_i := u_i+v_i \omega$ where $\omega := \frac{\underline{x}}{\|\underline{x}\|}$.
The slice monogenic theta series attached to $L$ with characteristic $w$ (depending on $\omega$) is defined by
\begin{equation}
{\Theta}_L(x,w) := \sum\limits_{q \in L} \exp_{*}\left((\pi |q|^2 x + 2 \pi  \langle q,w \rangle) \frac{\underline{x}}{\|\underline{x}\|}\right),
\end{equation}
where $\exp_{*}$ is the $*$-exponential function defined by \eqref{exp*} and where $ \langle q,w \rangle = \sum\limits_{i=0}^n q_i w_i$ is a $\mathbb{C}_{\omega}$-valued bilinear form.
\end{definition}
\begin{remark}{\rm
 We point out that $ \langle q,w \rangle$ is a bilinear form, but it is not a Hermitian inner product.
Moreover, we note that one can also define $ \langle q,w \rangle$ as  $\langle q,w \rangle = q' I {w'}^T$, where $q'=(q_0,\ldots, q_n)\in\mathbb{R}^{n+1}$, $w'=(w_0,\ldots, w_n)\in \mathbb{C}_{\omega}^{n+1}$, $I$ is the $n+1$-dimensional unit matrix, moreover one could replace $I$ by a matrix $S$ symmetric and positive definite with real entries, since a general lattice $L$ can be obtained by $A \mathbb{Z}^{n+1}$ with an invertible matrix $A$.}
\end{remark}
\begin{remark}
  {\rm Note that $\langle w,w \rangle=\sum_{i=0}^n w_i^2$, so it does not coincide with $|w|^2$. On the other hand, $\langle q,q \rangle =\sum_{i=0}^n q_i^2=|q|^2$ since $q_i\in\mathbb R$, $i=0,\ldots, n$.}
\end{remark}

\begin{proposition}\label{thetaconvergence}
The series ${\Theta}_L(x,w)$ converges normally on $H \times (\mathbb{C}_{\omega}^{n+1})$
\end{proposition}
\begin{proof}
We note that $H$ has no intersection with the $x_0$-axis, since $H = (\mathbb{R}^{n+1}) \setminus\mathbb R \cdot 1$.
Thus, as before, we write $x=x_0+\omega r\in H$, and $w = \sum\limits_{i=0}^n e_iw_i$ where each $w_i \in \mathbb{C}_{\omega}$ attached to the specific $\omega:= \frac{\underline{x}}{\|\underline{x}\|}$, i.e.  $w_i = u_i + v_i \omega$ with $u_i \in \mathbb{R}$ and $v_i > 0$ for all $i=0,\ldots,n$. 


We point out, one more time, that $\omega = \omega(x)$, and so also $w$, depends thus on the choice of $x$, in contrast to the classical complex case.

Let us next consider the functions $f(x)=x\pi |q|^2 \frac{\underline{x}}{\|\underline{x}\|}$, $g(x)=2 \pi \langle q,w \rangle \frac{\underline{x}}{\|\underline{x}\|}$.
First we note that
$$
f(x_0+\omega r)= (x_0+\omega r)\pi |q|^2\omega,\qquad
g(x_0+\omega r)= 2 \pi  \langle q,w \rangle \omega.
$$
The restrictions $f_{|H^+(\mathbb{R}^{n+1})\cap \mathbb{C}_\omega}$, $g_{|H^+(\mathbb{R}^{n+1})\cap \mathbb{C}_\omega}$
are in the kernel of the Cauchy-Riemann operator $1/2(\partial_{x_0}+\omega \partial_r)$ and so they are slice monogenic. In particular, $g$ is locally constant (constant on every $\mathbb C_\omega$). For any fixed $\omega\in\mathbb{S}^{n+1}$, both the functions have coefficients in the complex plane $\mathbb{C}_{\omega}$, $\omega=\underline{x}/\|\underline{x}\|$. Thus, they are commuting with respect to the $*$- product, so \eqref{equalityexp} holds. This is crucial.
Applying the Clifford norm introduced in Section~2 we get
\begin{eqnarray}\label{Cliffest}
\|\Theta_L(x,w)\| &\le & 1 +  \sum\limits_{q \in L \backslash\{0\}}
\Big\|	
\exp_*(\pi |q|^2(x_0+\omega r) \omega + 2 \pi  \langle q,w\rangle \omega)
\Big\|
\end{eqnarray}
Now, again in the Clifford norm we have
\begin{eqnarray*}
& & \Big\|\exp_*\Big( \pi \omega(|q|^2(x_0+r \omega) + 2 \sum\limits_{i=0}^n q_i w_i)   \Big)\Big\|\\
&=&  \Big\|\exp_*\Big( \pi \omega(|q|^2(x_0+r \omega) + 2 \sum\limits_{i=0}^n q_i (u_i + v_i \omega))   \Big)\Big\|\\
&=& \Big\| \exp_*\Big( \pi \omega(|q|^2 x_0 + 2 \sum\limits_{i=0}^n q_i u_i) - \pi(|q|^2 r + 2 \sum\limits_{i=0}^n q_i v_i)    \Big) \Big\| \\
&=& \Big\| \exp_*(\pi \omega(|q|^2 x_0 + 2 \sum\limits_{i=0}^n q_i u_i))\Big\| \cdot e^{-\pi(|q|^2r+2\sum\limits_{i=0}^n q_i v_i)},
\end{eqnarray*}
where we explicitly use the commutation property of the $*$-exponential in Theorem \ref{theoremcomm}.  Now, in view of the Euler formula, which follows by the definition \eqref{exp*} of $*$-exponential, we also have in our case that $\| \exp_*(\pi \omega(|q|^2 x_0 + 2 \sum\limits_{i=0}^n q_i u_i))\| = 1$, because $|q|^2 x_0 + 2 \sum\limits_{i=0}^n q_i u_i$ is real-valued.

Now, consider $w$ (for any fixed $\omega$) in a compact set covered by a ball of radius $\le r$. About the exponent of the second term we can now say that
$$
|q|^2 r + 2 \sum\limits_{i=0}^n q_i v_i \ge \frac{1}{2}|q|^2 r \ge \frac{1}{2} |q|^2 r_0
$$
holds for a fixed real positive $r_0 \le r$, except for finitely many lattice points $q$.

So, the whole series can be majorized by a multidimensional convergent geometric series of the form
$$
\sum\limits_{q \in L}(e^{-\frac{1}{2} \pi r_0})^{|q|^2}
$$
which in the particular case where $|q|^2 \in \mathbb{N}_0$ can also be directly expressed in terms of classical geometric series. In the other cases, one can consider $\lfloor |q|^2\rfloor$ where
$\lfloor \cdot\rfloor$ denotes the floor function.

Notice that the presence of the imaginary unit $\omega$ with $\omega^{2}=-1$ is crucial here for this whole argumentation.
\end{proof}
\begin{remark}
{\rm To perform our calculations on the right hand side of \eqref{Cliffest}, it is crucial that in the exponential $\exp_* (\alpha (x_0+\omega r) \omega')$ the multiplications of two imaginary units $\omega$ and $\omega'$ (which coincide) gives a real number. If $\omega'$ is any imaginary unit which does not belong to $\mathbb C_\omega$ then we do not obtain any real part.}
\end{remark}

Let us now turn to the appropriate definition of the theta series in the setting of the right half-space $H^{r} = \{x = x_0 + \underline{x}  \in \mathbb{R}^{n+1} \mid x_0 > 0\}$. We recall that $\omega:= \frac{\underline{x}}{\|\underline{x}\|}$ when $\underline{x}\not=0$ and $x=x_0+\omega r\in\mathbb C_\omega$ while this representation is not unique when $x\in\mathbb R$ since $x=x_0+0 \cdot \omega$ for any $\omega\in\mathbb S^{n-1}$.

\begin{definition}(Slice monogenic theta series associated with the right half-space $H^{r}$)\\
Let $x \in H^r$ and $w\in \mathbb{C}_{\omega}^{n+1}$ be of the form $w=\sum_{i=0}^n e_iw_i=\sum_{i=0}^n e_i(u_i + v_i \omega)$, $u_i$, $v_i\in\mathbb{R}$, where $\omega := \frac{\underline{x}}{\|\underline{x}\|}$ if $x\not\in\mathbb R$ or $\omega$ is any element in $\mathbb S^{n-1}$ if $x\in\mathbb R$.

Then the slice monogenic theta series associated with $H^r$ attached to $L$ with characteristic $w$ (depending on $\omega$) is defined by
\begin{equation}
{\Theta}^r_L(x,w) := \sum\limits_{q \in L} \exp_{*}(-\pi |q|^2 x + 2 \pi  \langle q,w \rangle \frac{\underline x}{\|\underline x\|}).
\end{equation}
\end{definition}

The convergence proof on $H^{r}$ can be done in formal analogy to the one presented for $H$. For completeness  we present it in detail:
\begin{proposition}
The series ${\Theta}^r_L(x,w)$ converges normally on $H^r \times (\mathbb{C}_{\omega}^{n+1})$.
\end{proposition}
\begin{proof}
We have
\begin{eqnarray*}
\|\Theta^r_L(x,w)\| &\le & \sum\limits_{q \in L} \Big\|\exp_*(-\pi|q|^2(x_0+r\omega) + 2\pi(\sum\limits_{i=0}^n q_i (u_i + v_i \omega)\omega)  \Big\| \\
&=& \Big\|\exp(-\pi |q|^2 x_0 - 2 \pi \sum\limits_{i=0}^n q_i v_i)\Big\|\cdot \underbrace{\Big\|\exp((-\pi |q|^2r+2\pi(\sum\limits_{i=0}^n q_iu_i))\omega)\Big\|}_{=1}\\
&=& e^{-\pi(|q|^2x_0 +2 \sum\limits_{i=0}^n q_i v_i)}.
\end{eqnarray*}
Now, considering $w$ in a compact set for any fixed $\omega$, we can again argue that for except of finitely many $q$ we can estimate
$$
|q|^2+2 \sum\limits_{i=0}^n q_i v_i \ge \frac{1}{2}|q|^2 x_0 \ge \frac{1}{2} |q|^2 {\mathcal{X}}_0
$$
for a positive ${\mathcal{X}}_0 \le x_0$. So the series can be majorized by the multidimensional geometric series $\sum\limits_{q \in L} (e^{-\frac{1}{2}\pi {\mathcal{X}}_0})^{|q|^2}$ and so it converges.
\end{proof}

\begin{remark}\label{rmk4:10}{\rm
Alternatively one might think to define a hypercomplex theta series in the following four ways but below we explain why these alternative definitions cannot work in the slice monogenic setting.
\begin{enumerate}
\item
Choose a $k \in \{1,\ldots,n\}$ and define
$$
\Theta(x,w) = \sum\limits_{q \in L} \exp_*(\pi |q|^2 x e_k + 2 \pi e_k  \langle q,w \rangle)
$$
Problem: This series converges on the space $H^+ := \{x \in \mathbb{R}^{n+1} \mid x_k > 0\}$, but this is not axially symmetric with respect to the real line and does not intersect it, so we cannot apply the classical tools of slice monogenic analysis.
\item One may define
$$
\Theta(x,w) = \sum\limits_{q \in L} \exp_*(\pi |q|^2 x i + 2 \pi i \langle q,w \rangle)
$$
where $i$ is the imaginary unit of the complexified Clifford algebra $\mathbb{C}_{n}=\mathbb{R}_{n} \otimes_{\mathbb{R}} \mathbb{C}$. \\
Problem: How to define slice monogenicity in $\mathbb{C}_{n}$ and to cope with the convergence domain  (we do not have a direction in which the function exponentially decreases since $xi$ does not have a nonzero real part).
\item In the same spirit we can consider
$$
\Theta(x,w) = \sum\limits_{q \in L} \exp_*(\pi |q|^2 x \tilde{e_0} + 2 \pi \tilde{e_0} \langle q,w \rangle)
$$
with an extra element $\tilde{e_0}$ outside the space $\mathbb{R}^{n}$, see \cite{CDR}. But, as in the previous case, also here the convergence property is spoiled and we do not get an exponential decrease.
\item We now define
$$
\Theta(x,w) = \sum\limits_{q \in L} \exp_*(-\pi |q|^2 x  + 2 \pi e_k \langle q,w \rangle)
$$
Problem: The expression $\exp(2 \pi \langle q,q \rangle)$ that will appear in the quadratic extension in the theta transformation formula is not equal to $1$, so the quadratic extension that will be applied in the transformation formula will not work; see the comment in the proof of the transformation formula in the next section.
\end{enumerate}
}
\end{remark}
\begin{remark} {\rm
Our definition fits canonically with the definition of Poincar\'e series that we worked out in our previous paper \cite{CKS2016} to which we refer for the notations.
For any positive integer $N \ge 3$ the slice monogenic Poincar\'e series is defined by
$$
P(x) = \sum\limits_{M: \Gamma_{RAV}^\infty[N] \backslash \Gamma_{RAV}[N]} (cx+d)^{-1} F(M \langle x \rangle ),
$$
where  $
F(x) := \exp_*(x \omega)$.}
\end{remark}
\subsection{Poisson summation and the transformation law}
In this subsection we give a proof for the theta transformation formula in the slice monogenic setting. The interested reader may consult Freitag's book \cite{Freitag} for the complex case version of the result. Due to a number of important peculiarities that have been used in the two appropriate definitions of the theta series, as well as special properties of the $*$-exponential function, we can prove a transformation formula also in the slice monogenic case.
First  we prove the formula in the setting of $H = \mathbb{R}^{n+1}\backslash \mathbb{R} = \bigcup_{\omega \in \mathbb{S}^{n-1}} \mathbb{C}_{\omega}^+$ . In this context it is important to note that since $\mathbb{C}_{\omega}^+$ is simply connected for any $\omega \in \mathbb{S}^{n-1}$ we can uniquely select one specific branch of the root $x^{\frac{n+1}{2}}$ (being the same for all $\omega \in \mathbb{S}^{n-1}$) so that all the expressions appearing in the following statement are well-defined.
\begin{theorem}\label{thetatrafoH} (Theta transformation formula in the setting of $H$).\\
For all $x \in H$ and all $w\in \mathbb{C}_{\omega}^{n+1}$ of the form $w=\sum_{i=0}^n e_i w_i=\sum_{i=0}^n e_i(u_i + v_i \omega)$, $u_i$, $v_i\in\mathbb{R}$, where $\omega := \frac{\underline{x}}{\|\underline{x}\|}$ we have the following generalization of the Jacobian theta series identity for the slice monogenic setting:
$$
\sum\limits_{q \in L} \exp_{*}\left({\pi \langle q+w,q+w \rangle x \frac{\underline{x}}{\|\underline{x}\|}}\right) = \Big(x \omega^{-1} \Big)^{-(n+1)/2} *{|\det(L)|} {\Theta}_{L^{\sharp}}(-x^{-1},w).
$$
In particular for $w=0$ we have
$$
\theta_L(x) = \Big(x \omega^{-1}  \Big)^{-(n+1)/2} * {|\det(L)|} {\theta}_{L^{\sharp}}(-x^{-1}).
$$
Here $|\det(L)|$ is again the modulus of the Gram determinant of the lattice generators $\mathfrak{Q}_0, \ldots,\mathfrak{Q}_n$ corresponding to the value of ${\rm vol}(\mathbb{R} \oplus \mathbb{R}^n/L)$ and $\theta_L(x) := \Theta_L(x,0)$ is the theta-null function.
\end{theorem}

\begin{proof} Consider the following auxiliary function:
$$
f_x(w)=f(w,x) := \sum\limits_{q \in L} \exp_* \left({\pi \langle q+w,q+w \rangle x \frac{\underline{x}}{\|\underline{x}\|}}\right).
$$
It is slice monogenic in $x$ by its definition.
Per construction we have $f_x(w+l) = f_x(w)$ for all $l \in L$, so $f$ is $L$-periodic and we also note that $f$ belongs to the  class $C^2$ and that the expression is integrable over $\mathbb R^{n+1}$.
We can write, using $ \langle q+w,q+w \rangle =\langle q,q\rangle+2\langle q,w\rangle+\langle w,w\rangle$ and the fact that all the summands have coefficients in $\mathbb{C}_\omega$:
$$
f(w,x)= \sum\limits_{q \in L} \exp_*(\pi \langle q+w,q+w \rangle (x_0+\omega r)\omega)= \sum\limits_{q \in L^{\sharp}} a_q(x) \exp(2\pi \langle q,w \rangle \omega),
$$
where $a_q(x)$ is $\mathbb C_{\omega}$-valued and slice monogenic.
We note that $f(w,x)$ is a multivariate Fourier-type series in the complex plane $\mathbb{C}_{\omega}$
with the commuting property:
\begin{eqnarray}\label{Fourier}
a_{q}(x) &=&  {|\det(L)|} \int_{[0,1]^{n+1}} f(w,x) \exp({-2 \pi  \langle q,w\rangle \omega}) du \\
 &=&  {|\det(L)|} \int_{[0,1]^{n+1}} \exp({-2 \pi  \langle q,w\rangle \omega}) f(w,x)  du ,\nonumber
\end{eqnarray}
where $w=u+\omega v$.
The expression $a_q(x)$ in formula \eqref{Fourier} is the Fourier transform $\mathcal{F}_{\omega}(f)$ performed on the complex plane $\mathbb C_{\omega}$ where $\omega\in\mathbb{S}^{n-1}$ (and it is fixed by $x$ that here works as a parameter). It is the classical Fourier transform where the imaginary unit $i$ of the complex numbers is replaced here by $\omega$. The $\omega\in\mathbb{S}^{n-1}$ is an element in the algebra and this is what makes the Fourier transform fully intrinsic (compare with Remark \ref{rmk4:10}, point 2).
Now, the function $f(w,x) \exp({-2 \pi  \langle q,w\rangle \omega})$ is slice monogenic. Therefore, if $x=x_0+\omega r$, then $f(w,x)$ is in the kernel of the Cauchy-Riemann operator of the complex plane $\mathbb C_{\omega}$. Consequently, the equality
$$f(w,x) e^{-2 \pi  \langle q,w\rangle \omega}=e^{-2 \pi  \langle q,w\rangle \omega} f(w,x) $$ leads to
\begin{eqnarray*}
a_q (x) &=&  {|\det(L)|} \int_{[0,1]^{n+1}}  f(w,x)e^{-2 \pi  \langle q,w\rangle \omega} du \\
   &=&  {|\det(L)|} \int_{[0,1]^{n+1}} \sum\limits_{q \in L} \exp_{*}({\pi \langle q+w,q+w \rangle (x_0+\omega r) \omega})  \exp({-2 \pi \langle q,w\rangle \omega} ) du.
\end{eqnarray*}
Due to the normal convergence of the series, see Proposition \ref{thetaconvergence}, one may interchange the integration process with the summation process so that the latter expression can be rewritten as:
$$
a_q(x) =  {|\det(L)|}\sum\limits_{q \in L} \int_{[0,1]^{n+1}}  \exp({-2 \pi  \langle q,w\rangle \omega}) \exp_{*}({\pi \langle q+w,q+w \rangle (x_0+\omega r) \omega} )  du.
$$
Next we apply a linear change of variable of the form $w \mapsto w -q$, leaving the differential invariant. This leads to
\begin{eqnarray*}
a_q (x_0+\omega r) &= &   {|\det(L)|}\sum\limits_{q \in L} \int_{[0,1]^{n+1}-q}
\exp({-2 \pi  \langle q,w-q\rangle \omega}) \exp_{*}({\pi (\sum_{i=0}^n w_i^2) (x_0+\omega r)\omega})   du\\
&= &   {|\det(L)|}\sum\limits_{q \in L} \int_{[0,1]^{n+1}-q}  \exp({-2 \pi  \langle q,w\rangle \omega}) \underbrace{e^{2 \pi  \langle q,q\rangle \omega}}_{=1}\exp_{*}({\pi  (\sum_{i=0}^n w_i^2)(x_0+\omega r) \omega}) du\\
&= &   {|\det(L)|}\sum\limits_{q \in L} \int_{[0,1]^{n+1}-q}   \exp({-2 \pi \langle q,w\rangle \omega}) \exp_{*}({\pi  (\sum_{i=0}^n w_i^2)(x_0+\omega r) \omega}) du.
\end{eqnarray*}
To apply this argument it was important (see Remark 3.9 Point 3) that $\omega$ is present in the definition, to make $\exp(2\pi \langle q,q \rangle \omega) = 1$ Without the $\omega$ this would not be true. Moreover it is crucial that $|q|^2\in\mathbb N_0$.
Note also that, although the standard $\exp({-2 \pi \langle q,w\rangle \omega})$ is in fact equal to  $\exp_*({-2 \pi \langle q,w\rangle \omega})$, for the computations below we need to compute the $*$-product and thus we use the first notations to emphasise that we work in the slice monogenic setting.
Next, we note that the functions $f(x)=-2 \pi \langle q,w\rangle \omega = -2 \pi \langle q,w\rangle \frac{\underline{x}}{\|\underline{x}\|}$ and $g(x)= \pi(\sum_{i=0}^n w_i^2) x \omega =  \pi(\sum_{i=0}^n w_i^2) x \frac{\underline{x}}{\|\underline{x}\|}$ are commuting with respect to the $*$-product so, by Theorem \ref{theoremcomm}, we have
\begin{eqnarray*}
& & \exp_*({-2 \pi \langle q,w\rangle \omega}) \exp_{*}(\pi(\sum_{i=0}^n w_i^2)(x_0+\omega r) \omega)\\
&=& \exp_*({-2 \pi \langle q,w\rangle \omega})* \exp_{*}({\pi(\sum_{i=0}^n w_i^2)(x_0+\omega r) \omega})\\
&=& \exp_{*}({(\pi (x_0+\omega r) (\sum_{i=0}^n w_i^2)-2\pi \langle q,w\rangle)\omega})\\
&=& \exp_{*}(\pi x ((\sum_{i=0}^n w_i^2)-2(x_0+\omega r)^{-1}\langle q,w\rangle)\omega)\\
\end{eqnarray*}

Hence we get
\[
\begin{split}
&a_q(x_0+\omega r) = {|\det(L)|} \int\limits_{]-\infty,+\infty[^{n+1}} \exp_{*}(\pi x ((\sum_{i=0}^n w_i^2)-2x^{-1}\langle q,w\rangle)\omega) du_0 \cdots du_n\\
&=
 {|\det(L)|} \int\limits_{]-\infty,+\infty[^{n+1}} \exp_{*}(\pi x ((\sum_{i=0}^n w_i^2)-2x^{-1} \langle q,w\rangle + x^{-2} |q|^2 - x^{-2}|q|^2)\omega) du_0 \cdots du_n\\
&=  {|\det(L)|} \exp_{*}(-\pi x^{-1}|q|^2\omega) \int\limits_{]-\infty,+\infty[^{n+1}} \exp_{*}(\pi x ((\sum_{i=0}^n w_i^2)-2x^{-1}\langle q,w\rangle +x^{-2}|q|^2)\omega) du_0 \cdots du_n\\
&=  {|\det(L)|} \exp_{*}(-\pi x^{-1}|q|^2\omega) \\
& \times \int\limits_{]-\infty,+\infty[^{n+1}} \exp_{*}(\pi x(\sum\limits_{i=0}^n w_i^2 - 2x^{-1} \sum\limits_{i=0}^n q_i w_i +x^{-2} \sum\limits_{i=0}^n q_i^2)
\omega) du_0 \cdots du_n.
\end{split}
\]
Notice here that $x$ is a general paravector from $\mathbb{R}^{n+1} \backslash \mathbb{R}$, so we here really deal with a hypercomplex expression. Next, the latter expression can be written as
\begin{equation}
a_q(x_0+\omega r)
=  {|\det(L)|} \exp_{*}(-\pi x^{-1}|q|^2\omega)
\times \Bigg(\prod\limits_{i=0}^n \int\limits_{-\infty}^{+\infty}\exp_{*}(\pi x (w_i^2-2x^{-1} q_i w_i + x^{-2} q_i^2)\omega) du_i \Bigg).
\end{equation}
Now put $x = r \omega$, $x_0=0$.
Then the latter equation becomes
$$
a_q =  {|\det(L)|} \exp_{*}(-\pi (r \omega)^{-1}|q|^2\omega) \times \prod\limits_{i=0}^n \Bigg(
\int\limits_{-\infty}^{+\infty} \exp_*(\pi r \omega(w_i^2+ \frac{2}{r}\omega q_i w_i - \frac{1}{r^2}q_i^2)\omega) du_i
\Bigg).
$$
Let us now decompose each $w_i \in \mathbb{C}_{\omega}$ in the form
$$
w_i = u_i + v_i \omega.
$$
Using this decomposition we can rewrite each term $$w_i - x^{-1} q_i = w_i + \frac{1}{r} \omega q_i$$ in the form   $u_i + v_i \omega
+ \frac{1}{r}\omega q_i$ for all $i=0,\ldots,n$. If we choose the $\omega$-part of each $w_i$ in the way $v_i = - \frac{1}{r} q_i$ then the expressions $t_i := w_i + \frac{1}{r} \omega q_i$ turn out to be real positive. So we can rewrite each term
$$
(w_i^2 + \frac{2}{r} \omega q_i w_i - \frac{1}{r^2}q_i^2) = (w_i+\frac{1}{r} \omega q_i)^2 =:t_i^2
$$
with a positive real parameter $t_i > 0$.
So the expression for the Fourier coefficient can be re-expressed as
\begin{eqnarray*}
a_q &=&  {|\det(L)|} \exp_{*}(-\pi (r \omega)^{-1}|q|^2 \omega) \times \Bigg(\prod\limits_{i=0}^n \int\limits_{-\infty}^{+\infty} \exp((\pi r \omega t_i^2)\omega) dt_i \Bigg)\\
&=&  {|\det(L)|} \exp_{*}(-\pi (r \omega)^{-1}|q|^2 \omega) \times \Bigg(\prod\limits_{i=0}^n \int\limits_{-\infty}^{+\infty} \exp(- \pi r  t_i^2) dt_i \Bigg).\\
& = &  {|\det(L)|} \exp_{*}(-\pi (r \omega)^{-1}|q|^2 \omega) \frac{1}{r^{\frac{n+1}{2}}}.
\end{eqnarray*}

So, for $x=\omega r$ we have obtained that
$$
\sum\limits_{q \in L} e^{\pi \langle q+w,q+w \rangle (\omega r) \omega} = r^{-(n+1)/2}  {|\det(L)|}
\sum\limits_{q \in L^{\sharp}} e^{\pi  |q|^2(-r)^{-1} \omega + 2 \pi  \langle q,w\rangle \omega}.
$$
In view of $x = r \omega$ we may identify $r$ by $x \omega^{-1}$.
In order to proceed further, we need to apply now a particular argument from slice monogenic function theory. The particular identity theorem for slice monogenic functions from \cite{altavilla} allows us to conclude from the previous line that we can then substitute $r \omega$ by $x_0 + r \omega$ so that actually
$$
\sum\limits_{q \in L} \exp_*({\pi \langle q+w,q+w \rangle  x \omega}) = (x \omega^{-1})^{-(n+1)/2} * {|\det(L)|}
\sum\limits_{q \in L^{\sharp}} \exp({\pi  |q|^2(-x)^{-1} \omega + 2 \pi  \langle q,w\rangle \omega})
$$
is true for all $x \in H$. It is clear that the left hand-side is slice monogenic. Also the right hand-side is slice monogenic by construction. The version of the identity theorem from \cite{altavilla} allows us to conclude the equality. We actually may observe that the $*$-product on the right-hand side is not necessary, hence we omit it in all that follows.

Note further that particularly, putting $\theta_L(x) := \Theta_L(x,0)$, we get
$$
\theta_L(x) =  {|\det(L)|} (x \omega^{-1})^{-(n+1)/2} \theta_{L^{\sharp}}(-x^{-1}).
$$
This completes the proof.
\end{proof}
We recall that in the case of a general lattice  we could alternatively also re-define $ \langle q+w,q+w \rangle$ as $(q+w)' S (q+w)$ where $S$ is an $(n+1) \times (n+1)$ positive symmetric matrix with real entries. In the case $S=I$ one then again obtains $\langle q+w,q+w \rangle =
\sum\limits_{i=0}^n (q_i + w_i)^2 = \sum\limits_{i=0}^n (q+w)_i^2$.
\begin{remark} {\rm
In the case $n=1$, the function $\theta_L(x)$ coincides with the classical theta null series in one complex variable attached to a two-dimensional lattice $L_2$,
cf. {\rm \cite{freitagb} p. 359}, namely with
$$
\theta_{L_2}(z) = \sum\limits_{q \in L_2} e^{\pi i |q|^2 z},
$$
transforming as
$$
\theta_{L_2}(z) =  {|\det(L_2)|} i z^{-1} \theta_{{L_2}^{\sharp}}(-\frac{1}{z}),
$$
see also {\rm Theorem VI.4.7} of {\rm \cite{freitagb}}.
To get the very classical theta series one has to extend the summation only over the one-dimensional lattice $\mathbb{Z}$ instead of extending it over $L_2$. The series $\theta_{\mathbb{Z}}(z)$ then transforms as $\theta_{\mathbb{Z}}(z) = (-iz)^{-1/2} \theta_{\mathbb{Z}}(-1/z)$.}
\end{remark}

The setting of the right half-space may be established by similar lines of arguments. Note that to ensure the invariance of the right half-space $H^r$ we have to consider, as suggested by A. Krieg in \cite{Krieg1988}, the modified inversion $x \mapsto x^{-1}$, because the usual Kelvin inversion $x \mapsto -x^{-1}$ does not preserve $H^r$.
\begin{theorem}\label{thetatrafoHr}(Theta transformation formula in the setting of $H^r$).\\
For all $x \in H^r = \{z \in \mathbb{R}^{n+1} \mid x_0 > 0\}$ and all $w \in \mathbb{C}_{\omega}^{n+1}$ where $\omega := \frac{\underline{x}}{\|\underline{x}\|}$ if $x \not \in \mathbb{R}$ and where $\omega$ can be chosen freely if $x \in \mathbb{R}^{>0}$  we have that
$$
\sum\limits_{q \in L} \exp_*(-\pi \langle q+w,q+w \rangle x) = x^{-(n+1)/2} | {\det(L)}| \Theta^r_{L^{\sharp}}(x^{-1},w),
$$
and particularly for $w=0$ we have, setting $\theta^r_L(x):=\Theta^r_L(x,0)$, $\theta^r_L(x)=x^{-(n+1)/2} | {\det(L)}| \theta^r_{L^{\sharp}}(x^{-1})$.
\end{theorem}

\begin{proof}
In this setting we now define analogously
$$
f_x(w)=f(w,x) := \sum\limits_{q \in L} \exp_*(-\pi \langle q+w,q+w \rangle x).
$$
Again here we have $f(w+l)=f(w)$ for all $l \in L$, so also $f$ can be expanded in a Fourier series of the form $\sum\limits_{q \in L^{\sharp}} a_q(x) \exp(2 \pi \langle q,w \rangle \omega)$ with $a_q(x) =  {|\det(L)|} \int\limits_{[0,1]^{n+1}}f(w,x) e^{-2\pi \langle q,w \rangle \omega}$. By applying the same argumentation with the shift $w \mapsto w -q$ as in Theorem~\ref{thetatrafoH} we get using the decomposition $w = u + \omega v$:
\begin{eqnarray*}
a_q(x)
&=&  {|\det(L)|} \sum\limits_{q \in L} \int\limits_{[0,1]^{n+1} - q}\exp (-2 \pi \langle q,w \rangle \omega) \exp_*(-\pi (\sum\limits_{i=0}^n w_i^2)(x_0+r\omega))du.
\end{eqnarray*}
Now in the setting of $H^r$ the exponential expressions can be rewritten in the form
\begin{eqnarray*}
&  & \exp_*(-\pi(x_0+\omega r)(\sum\limits_{i=0}^n w_i^2) - 2 \pi \langle q,w \rangle \omega)\\
&= & \exp_*(-\pi(x_0+\omega r)\Big( \sum\limits_{i=0}^n w_i^2+2(x_0+r \omega)^{-1} \langle q,w \rangle \omega \Big))\\
&=&  \exp_*(-\pi x \Big( \sum\limits_{i=0}^{n} + 2x^{-1} \langle q,w \rangle \omega   \Big))\\
& = & \exp_*(\pi x \Big( \sum\limits_{i=0}^{n} - 2x^{-1} \langle q,w \rangle \omega   \Big)\omega^2)\\
&=& \exp_*(\pi x\Big(\sum\limits_{i=0}^n w_i^2 - 2x^{-1} \langle q,w \rangle \omega +x^{-2}|q|^2-x^{-2}|q|^2        \Big)\omega^2)\\
&=& \exp_*(\pi x(x^{-2}|q|^2)\omega^2) \exp(\pi x\Big(\sum\limits_{i=0}^n w_i^2-2x^{-1} \langle q,w \rangle \omega -x^{-2}|q|^2 \Big)\omega^2) \\
&=&\exp_*(-\pi (x^{-1}|q|^2)) \exp(\pi x\Big(\sum\limits_{i=0}^n w_i^2-2x^{-1} \langle q,w \rangle \omega +(x\omega)^{-2}|q|^2 \Big)\omega^2).
\end{eqnarray*}
Now take again $x = r \omega$. Then we can again  adjust the $v$-part of $w$ such that $w_i^2-2(r \omega)^{-1} q_i w_i\omega +r^{-2}q_i^2$ is a real positive entity that can be identified by $t_i^2$ with $t_i \in \mathbb{R}$.  So, finally one may obtain in this setting that
$$
a_q(r \omega) =   \exp_*(-\pi ((r \omega)^{-1}|q|^2)) \int\limits_{]-\infty,+\infty[^{n+1}} \exp(-\pi r \omega t^2) dt = \exp_*(-\pi ((r \omega)^{-1}|q|^2)) (\frac{1}{ {r\omega}})^{n+1},
$$
from which the stated identity follows after the application of the version of the identity theorem presented in \cite{bookfunctional,slicecss}.
\end{proof}

\begin{remark} {\rm
Note that in the monogenic setting, in general the transformation
$$
F(x) := \frac{\overline{x}}{|x|^a} f(-x^{-1})
$$
only preserves the monogenicity property if particularly $a=n+1$. Now in the slice monogenic case the expression $$(x\omega^{-1})^{-(n+1)/2}f(-x^{-1})$$ remains slice monogenic for any integer $n$. This allows us to associate to every lattice a slice monogenic theta series. A monogenic or $k$-hypermonogenic automorphic form with a transformation behavior of the form $f(x) = (x\omega^{-1})^{-(n+1)/2}f(-x^{-1})$ can probably not be found. Among all existing hypercomplex function theories the slice monogenic setting seems to be the only setting in which the construction of direct generalizations of theta series is possible. Notice also that the identity theorem offered by the theory of slice monogenic functions is much stronger than the identity theorem of monogenic functions. In the slice monogenic setting the coincidence of function values of two functions along one line already yields the identity over the whole space while in monogenic function theory one requires, in general, the coincidence on an $n$-dimensional submanifold.
}
\end{remark}

An interesting question arises around the periodicity of the theta series. While the slice monogenic theta series $\Theta(x,w)$ associated with $H = \mathbb{R}^{n+1}\backslash \mathbb{R}$ are $1$-fold periodic in $x$ with respect to the $x_0$-part, in association with its version for  the right half-space $H^r$ we do not have the usual periodicity but a radial periodicity in the reduced vector part $x_1 e_1 + \cdots + x_n e_n$. This will be explained in the following proposition. More precisely, we can say

\begin{proposition} (Periodicity properties).
Both theta series $\Theta_L(x,w)$ and $\Theta^r_L(x,w)$ are $n+1$-fold periodic with respect to the lattice $L$; in the second variable $w$  we have
$$
\Theta_L(x,w+l) = \Theta_L(x,w) \quad\quad \text{ for\;all}\ l \in L.
$$
and
$$
\Theta^r_L(x,w+l) = \Theta^r_L(x,w) \quad\quad \text{ for\;all}\ l \in L.
$$
Furthermore, $\Theta_L(x,w)$ satisfies $\Theta_L(x+2,w) = \Theta_L(x,w)$ for all $x,w$.

Writing the reduced vector part of $x$ in polar form, i.e. writing $x = x_0 + r \omega$ where as usual $\omega:= \frac{\underline x}{ {\|\underline x\|}}$ with a positive $r > 0$, then for the other kind of theta series $\Theta^r_L(x,w)$ we observe the following radial periodicity concerning the first variable of the form
$$
\Theta^r_L(x_0 + r \omega,w) = \Theta^r_L(x_0 +(r+2)\omega,w).
$$
\end{proposition}
To the proof, one observes the $L$-periodicity in the variable $w$ by a direct rearrangement argument. In the case of $\Theta_L(x,w)$ the periodicity in the $x_0$-direction is also readily seen from its definition.

The radial periodicity in the reduced vector part of the first variable of $\Theta^r_L(x,w)$ is inherited by the radial periodic property of the slice monogenic $*$-exponential function $\exp_*$.

\par\medskip\par

\begin{remark} {\rm
In summary, our slice monogenic theta series are examples of $1$-fold-periodic or radially periodic quasi-modular forms on $H$ or $H^r$ , respectively. They are quasi-invariant under the inversion $x \mapsto -x^{-1}$, or the modified inversion $x \mapsto x^{-1}$ up to the automorphic factor
$\left(x \left(\frac{\underline{x}}{\|\underline{x}\|}\right)^{-1}\right)^{-\frac{n+1}{2}}$ or $x^{-\frac{n+1}{2}}$, respectively, and exhibit the above stated periodic or radially periodic behavior, respectively. The slice monogenic Eisenstein  and Poincar\'e series that we discussed in \cite{CKS2016} were also quasi-invariant under the inversion (but with a different automorphy factor) and they were invariant under translations in the $x_0$-direction like the theta series in the context of $H$.}
\end{remark}

\subsection{The conjugated theta functions}
For simplicity we now focus on the setting of $H$ from now on. All results presented in the sequel can be directly translated to the setting of $H^r$ when replacing the corresponding theta functions.

As previously introduced, we defined the slice monogenic theta-null function associated to an $n+1$-dimensional lattice $L$ as
$$
\theta(x):= \Theta_L(x,0) = \sum\limits_{q \in L} \exp_*\left({\pi |q|^2x \frac{\underline{x}}{\|\underline{x}\|}}\right),
$$
where we consider the same conditions on $L$ as in the beginning of Section~4.1. In particular we assumed that $|q|^2 \in \mathbb{N}_0$ which means that $L$ is supposed to be integral. To leave it simple we furthermore assume that $L$ is unimodular in all that follows in this subsection and the following one. We could perform the following considerations more generally, but then one has to use the dual lattice. In the case of unimodularity we simply have $L^{\sharp}=L$.

Now we introduce the following conjugated theta functions and study their invariance behavior. In turn these functions can be used as building blocks to construct slice monogenic quasi-modular forms.
\par\medskip\par
To introduce them, consider a system of representatives denoted by ${\cal{V}}(\frac{1}{2}L / L)$ of the quotient lattice $\frac{1}{2}L / L$.
A canonical choice is to take these representatives out of the set
$$
\{m_0 \mathfrak{Q}_0 + \ldots + m_n \mathfrak{Q}_n \}\quad{\rm with}\; 0 \le m_0 < 1,\;m_i \in \Big\{0,\frac{1}{2}\Big\},\;i=0,\ldots,n.
$$
For a fixed element $\tilde{q} \in {\cal{V}}(\frac{1}{2}L/L)$ we define the {\it first conjugated theta function} $\tilde{\theta}_{\tilde{q}}(x)$ as
\begin{eqnarray*}
\tilde{\theta}_{\tilde{q}}(x) & := & \Theta_L(x,\tilde{q}) \\
& = & \sum\limits_{q \in L} \exp_*({(\pi |q|^2x+2 \pi \langle q,\tilde{q}\rangle   )\omega})\\
& = & \sum\limits_{q \in L} \exp_*({\pi |q|^2 x \omega}) \cdot e^{2 \pi \langle q,\tilde{q}\rangle \omega}\\
& = &\sum\limits_{q \in L} \chi(q) \exp_*({\pi |q|^2 x \omega}),
\end{eqnarray*}
where
$$
\chi(q) = \left\{ \begin{array}{ll} 1 & {\rm if}\; |q|^2 \; {\rm even} \\
-1 & {\rm if} \; |q|^2 \; {\rm odd}  \end{array} \right.
$$
The definition of the {\it second conjugated theta function}
$\tilde{\tilde{\theta}}_{\tilde{q}}(x)$ is formally motivated by the theta transformation formula that we proved in the previous subsection.

We have
$$
\sum\limits_{q \in L} \exp_*^{\pi \langle q+w,q+w \rangle  x \omega} = (x \omega^{-1})^{*-\frac{n+1}{2}} {|\det(L)|} \Theta_L(-x^{-1},w).
$$
Notice that in contrast to the classical complex case, we cannot choose $w$ completely freely, because in our case $w=w(x)$ since we have to choose $w \in {\mathbb{C}}_{\omega}^{n+1}$. However, making nevertheless formally the substitution $w = \tilde{q}$ with $\tilde{q} \in {\cal{V}}(\frac{1}{2}L/L)$ motivates the definition:
$$
\tilde{\tilde{\theta}}_{\tilde{q}}(x) := \sum\limits_{q \in L} e^{\pi |q+\tilde{q}|^2 x \omega}.
$$
By construction this function then satisfies the relation:
\begin{equation}\label{theta1trafo}
\tilde{\tilde{\theta}}_{\tilde{q}}(x) = (x \omega^{-1})^{-\frac{n+1}{2}}  {|\det(L)|}  \tilde{\theta}_{\tilde{q}}(-x^{-1}).
\end{equation}
If we replace $x$ by $-x^{-1}$ in equation~(\ref{theta1trafo}) then we also get
$$
\tilde{\tilde{\theta}}_{\tilde{q}}(-x^{-1})  =  (-x^{-1} \omega^{-1})^{-\frac{n+1}{2}} {|\det(L)|} \tilde{\theta}_{\tilde{q}}(x)
$$
which is equivalent to
$$
\tilde{\theta}_{\tilde{q}}(x) = (x^{-1}(-\omega^{-1}))^{\frac{n+1}{2}} \frac{1}{ {|\det(L)|}} \tilde{\tilde{\theta}}_{\tilde{q}}(-x^{-1}).
$$
So, we arrived at
\begin{equation}\label{theta2trafo}
\tilde{\theta}_{\tilde{q}}(x) = (x \omega^{-1})^{-\frac{n+1}{2}} \frac{1}{ {|\det(L)|}} \tilde{\tilde{\theta}}_{\tilde{q}}(-x^{-1}).
\end{equation}
We end this subsection by giving the following more general definition of the conjugated theta functions (where we involve the general parameter $w$, which however in contrast to the classical complex variable case is fixedly related to $\omega$ depending on $x$).
\begin{definition} Let $L$ be a general $(n+1)$ dimensional integral lattice in $\mathbb{R}^{n+1}$.

For all $x \in H$ and $w \in \mathbb{C}_{\omega}^{n+1}$ we define
$$
{\tilde{\Theta}}_L(x,w) = \sum\limits_{q \in L} \chi(q) \exp_{*}({\pi  |q|^2 x \omega}) e^{2 \pi \langle q,w\rangle \omega}
 $$
and
$$
{\tilde{\tilde{\Theta}}}_L(x,w) = \sum\limits_{q \in \frac{1}{2} L \backslash L} \exp_{*}({\pi  |q|^2 x \omega})
e^{2 \pi  \langle q,w\rangle \omega}
$$
where
$$
\chi(q) = \left\{ \begin{array}{ll} 1 & {\rm if}\; |q|^2 \; {\rm even} \\
                                   -1 & {\rm if} \; |q|^2 \; {\rm odd}  \end{array} \right.
$$
\end{definition}
Notice that whenever $|q|^2$ is even then we have $e^{\pi  |q|^2 \omega} = 1$. If $|q|^2$ is odd then $e^{\pi  |q|^2 \omega} = -1$.
\par\medskip\par
For $w=0$ we re-obtain the particular conjugated theta functions  $\tilde{\theta}_{\tilde{q}}(z)$ and $\tilde{\tilde{\theta}}_{\tilde{q}}(z)$.

\begin{proof} In view of $|\chi(q)|=1$ we can majorize the series ${\tilde{\Theta}}_L(x,w)$ by the series of the moduli of the theta series ${\Theta}_L(x,w)$. Concretely speaking, we have
$$
\|{\tilde{\Theta}}_L(x,w)\| \le \sum\limits_{q \in L} \|\exp_{*}({\pi |q|^2 x \omega}) \cdot e^{2 \pi \langle q,w\rangle \omega}\| \le C < \infty,
$$
for any pair $(x,w)$ belonging to a compact subset of $H  \times \mathbb{C}_{\omega}^{n+1}$ where we use the convergence argument applied in Proposition~\ref{thetaconvergence}.

\par\medskip\par

Similarly, we may argue that ${\tilde{\tilde{\Theta}}}_L(x,w)$ converges normally on $H  \times \mathbb{C}_{\omega}^n$, since this series is majorized by the theta series
$$
\Theta_{\frac{1}{2}L} (x,w) = \sum\limits_{q \in \frac{1}{2} L} \exp_{*}({\pi |q|^2 x \omega}) e^{2 \pi \langle q,w\rangle \omega}
$$
which is nothing else than the slice monogenic theta series for the larger lattice $\frac{1}{2}L$. This series over the larger lattice is still convergent according to the statement of Proposition~\ref{thetaconvergence}, because it guarantees the convergence for any arbitrary $n+1$-dimensional lattice, so in particular for $\frac{1}{2}L$.
\end{proof}

\subsection{Slice monogenic generalization of the Dedekind eta function and the modular discriminant}
In this section, we apply the slice monogenic conjugated theta functions that we defined in the previous section to introduce a slice monogenic generalization of the third power of the Dedekind $\eta$-function and a generalization of the modular discriminant $\Delta$ which also represented a central missing piece in the hypercomplex setting of automorphic forms.
\begin{definition}
Let $x \in H$ and $L$ be an $n+1$-dimensional unimodular lattice. Furthermore, fix a representative $\tilde{q}$ from $\frac{1}{2} L/L$.

Now we define the function
\begin{equation}\label{etadef}
\tilde{\eta}_{\tilde{q}}(x) := \theta(x) * \tilde{\theta}_{\tilde{q}}(x) * \tilde{\tilde{\theta}}_{\tilde{q}}(x),
\end{equation}	
where $*$ again is the star product of slice monogenic functions.
\end{definition}
This function generalizes up to a constant the third power of the Dedekind eta function which is a quasi-modular form of weight $1/2$. In our case we shall see that $\tilde{\eta}_{\tilde{q}}(x)$ is a slice monogenic quasi-modular form of weight $\frac{3(n+1)}{2}$. More precisely we will show:
\begin{theorem}
Let $L$ be an $n+1$-dimensional unimodular lattice in $\mathbb{R}^{n+1}$ and $\tilde{q}$ be a representative from $\frac{1}{2} L/L$. Then the above defined function $\tilde{\eta}_{\tilde{q}}(x)$
satisfies for each $x \in H$ the following transformation formula	
\begin{equation}\label{etatrafo}
\tilde{\eta}_{\tilde{q}}(x) = (x \omega^{-1})^{-\frac{3(n+1)}{2}}  {|\det(L)|} \tilde{\eta}_{\tilde{q}}(-x^{-1}).
\end{equation}
\end{theorem}
\begin{proof}
	To show this transformation behavior we apply the transformation formulas (\ref{theta1trafo}) and (\ref{theta2trafo}) in the definition (\ref{etadef}). Precisely speaking, we have
	\begin{eqnarray*}
    \tilde{\eta}_{\tilde{q}}(x) &=& \theta(x) * \tilde{\theta}_{\tilde{q}}(x) * \tilde{\tilde{\theta}}_{\tilde{q}}(x)\\
    &=& (x \omega^{-1})^{-\frac{n+1}{2}}  {|\det(L)|} \theta(-x^{-1})\\
    & &  *  (x \omega^{-1})^{-\frac{n+1}{2}} \frac{1}{ {|\det(L)|}} \tilde{\tilde{\theta}}_{\tilde{q}}(-x^{-1})\\
    & &  *  (x \omega^{-1})^{-\frac{n+1}{2}}  {|\det(L)|} \tilde{\theta}_{\tilde{q}}(-x^{-1})\\
    &=& (x \omega^{-1})^{-\frac{3(n+1)}{2}}  {|\det(L)|}
    \Big( \theta(-x^{-1}) *   \tilde{\tilde{\theta}}_{\tilde{q}}(-x^{-1})*  \tilde{\theta}_{\tilde{q}}(-x^{-1})\Big)\\
    &=& (x \omega^{-1})^{-\frac{3(n+1)}{2}}  {|\det(L)|} \Big( \theta(-x^{-1}) * \tilde{\theta}_{\tilde{q}}(-x^{-1})*   \tilde{\tilde{\theta}}_{\tilde{q}}(-x^{-1})\Big)\\
    &=& (x \omega^{-1})^{-\frac{3(n+1)}{2}}  {|\det(L)|}\tilde{\eta}_{\tilde{q}}(-x^{-1}).
	\end{eqnarray*}
In the proof we used the property
$$
 \tilde{\tilde{\theta}}_{\tilde{q}}(-x^{-1})* \tilde{\theta}_{\tilde{q}}(-x^{-1})=
\tilde{\theta}_{\tilde{q}}(-x^{-1})*   \tilde{\tilde{\theta}}_{\tilde{q}}(-x^{-1})
$$
which follows from the fact that the two theta functions do commute on the complex plane $\mathbb C_\omega$.
\end{proof}

\begin{remark} {\rm
	In the case of extending the summation only over a one-dimensional lattice, the resulting function then satisfies the transformation behavior
	$$
	\tilde{\eta}_{\tilde{q}}(x) = (x \omega^{-1})^{-\frac{3}{2}}  {|\det(L)|} \tilde{\eta}_{\tilde{q}}(-x^{-1})
	$$
	If we substitute $L$ by $\mathbb{Z}$ and if we consider a complex variable (case $n=1$), then we get the usual transformation behavior of the third power of the Dedekind eta function of the form
	$$
	\eta^3(z) = (-zi)^{-3/2} \eta^3(-\frac{1}{z}).
	$$}
\end{remark}

With these tools in hand we can finally define the slice monogenic modular discriminant function:
\begin{definition}
	Let $x \in H$ and $L$ be an $n+1$-dimensional unimodular lattice. Furthermore, fix a representative $\tilde{q}$ from $\frac{1}{2} L/L$.
	Then the slice monogenic associated modular discriminant can be defined as
	$$
	\triangle_{L,\tilde{q}}(x) := \Big(\tilde{\eta}_{\tilde{q}}(x)\Big)^{*8} = \Big(
	\theta(x) * \tilde{\theta}_{\tilde{q}}(x) * \tilde{\tilde{\theta}}_{\tilde{q}}(x)
	  \Big)^{*8}.
	$$
\end{definition}
In view of the transformation formulas (\ref{theta1trafo}) and (\ref{theta2trafo}) we can directly establish that the slice monogenic discriminant transforms like
\begin{eqnarray*}
\triangle_{L,\tilde{q}}(x)&=& (x \omega^{-1})^{-12(n+1)} { {|\det(L)|}}^8 \Big(
\theta(-x^{-1}) * \tilde{\theta}_{\tilde{q}}(-x^{-1}) * \tilde{\tilde{\theta}}_{\tilde{q}}(-x^{-1})
\Big)^{*8}\\
&=&  (x \omega^{-1})^{-12(n+1)}  |\det(L)|^4 \triangle_{L,\tilde{q}}(-x^{-1}).
\end{eqnarray*}
If we again substitute $L$ by $\mathbb{Z}$ and if put $n=1$ then we obtain the well-known usual relation of the classical discriminant
$$
\triangle(z) = z^{-12}\triangle\Big(-\frac{1}{z}\Big).
$$
\subsection{Differential equations}
The next theorem shows that the theta series $\Theta_L(x,w)$ satisfies the following  heat equation
which involves the slice derivative:
$$
[\Delta_w - c \partial_{s,x}] f(x,w) = 0.
$$
Here $\Delta_w$ stands for the Euclidean Laplacian $\Delta_w := \sum\limits_{i=0}^n \frac{\partial^2}{\partial w_i^2}$ and $\partial_{s,x}$ for the slice derivative with respect to $x$, denoted for short $\partial_s$, (for the definition see for example \cite{bookfunctional}).
However, in contrast to the usual heat equation, this version here involves the slice monogenic derivative instead of the usual temporal partial derivative $\partial_t$. Precisely, we have
\begin{theorem}
The theta series $\Theta_L(x,w)$ satisfy the differential equation
$$
[\Delta_w - 4 \pi \omega  \partial_s] \Theta_L(x,w) = 0
$$
\end{theorem}
{\it Proof}. The proof can be done by a direct computation. Clearly, for any $i=0,\ldots, n$ we have:
$$
\frac{\partial^2}{\partial w_i^2} \Theta_L(x,w) = \sum\limits_{q \in L} (2 \pi \omega q_i)^2 \exp_{*}({\pi  |q|^2 x \omega+2 \pi  \langle q,w\rangle \omega}),
$$
so that in summary one gets that
$$
\Delta_w \Theta_L (x,w) = \sum\limits_{q \in L} (2 \pi \omega)^2 |q|^2 \exp_{*}({\pi  |q|^2 x \omega +2 \pi  \langle q,w\rangle \omega}) .
$$
On the other hand one obtains by applying the slice monogenic derivative the expression
$$
\partial_s \Theta (x,w) = \sum\limits_{q \in L} (\pi \omega) |q|^2 \exp_{*}({\pi  |q|^2 x \omega +2 \pi  \langle q,w \rangle \omega})
$$
which coincides with the expression in the preceding line after multiplying it with the constant $(4 \pi \omega)$.  The stated result is proven.
\par\medskip\par
\begin{remark}{\rm By a similar direct computation we obtain that also the other slice monogenic theta series are solutions to the slice heat equation.}
\end{remark}

\subsection{A monogenic generalized theta function}

When we apply Fueter's theorem to the slice monogenic theta series that we introduced before then we obtain a generalization of theta series in the monogenic setting. To leave it simple we explain this procedure explicitly in the quaternionic case looking at the right half-space model working with $H^r = \{x \in \mathbb{H} \mid x_0 > 0\}$ since in this setting Fueter's theorem can be applied directly. Again, for simplicity we discuss the case $w=0$ addressing monogenic generalizations of the theta null series.

Before we can give the definition we first need to recall the definition of the following axially monogenic exponential function.

For any integer $k\geq 2$ consider the monogenic functions
$$f_k(x)=\Delta
(x^k)=\overset{\sim}f_k(x)=-4\displaystyle\sum_{j=1}^{k-1}(k-j)x^{k-j-1}\overline{x}^{j-1},$$
and then, for any $n\in\mathbb N_0$ define
$$
Q_n(x)=-\frac{f_{n+2}(x)}{(n+2)(n+1)}, \qquad n\in\mathbb N_0.
$$
In more explicit terms, we have
\begin{equation} \label{Qk2}
\displaystyle \mathcal{Q}_k(x)=\sum_{j=0}^k T^k_j x^{k-j}\overline{x}^j
\end{equation}
where $$\displaystyle
T^k_j:=T^{k}_{j}(3)=\frac{k!}{(3)_k}\frac{(2)_{k-j}(1)_j}{(k-j)!j!}=\frac{2(k-j+1)}{(k+1)(k+2)}$$
and $(a)_n=a(a+1)...(a+n-1)$ is the Pochhammer symbol.

For $x\in \mathbb{H}$ let
$${\rm
{Exp}}(x):=\sum_{k=0}^\infty\frac{Q_k(x)}{k!}$$ be the generalized
Cauchy-Fueter regular exponential function, see \cite{CMF2011}.

In this setting we can now define
\begin{definition} Let $x \in \mathbb{H}^r = \{x \in \mathbb{H} \mid x_0 > 0\}$ and let $L \subset \mathbb{H}$ be an arbitrary lattice with $|q|^2  \in \mathbb{N}_0$ for all $q \in L$. Then the attached monogenic theta function is defined by
$$
\theta_L^M(x) = \Delta_x\Big[\sum\limits_{q \in L} \exp_*(-\pi |q|^2 x)   \Big] = \sum\limits_{q \in L} {\rm Exp}(-\pi |q|^2 x),
$$
where ${\rm Exp}(x)$ is the axially monogenic exponential function introduced before.
\end{definition}
The theta transformation formula that we derived in Theorem \ref{thetatrafoHr} for the slice monogenic theta series will provide us with a functional equation for the monogenic theta function when applying Fueter's theorem to both sides of the equation. We prove:
\begin{theorem}(Functional equation for the monogenic theta function).\\
For each $x \in \mathbb{H}^r = \{x \in \mathbb{H} \mid x_0 > 0\}$ the monogenic theta function $\theta_L^M(x)$ satisfies the functional equation
\begin{eqnarray} \label{monthetatrafo}
\theta_L^M(x) &=& |\det(L)| \Bigg( x^{-\frac{n+1}{2}} \Delta_x \Big[\theta_{L^{\sharp}}(x^{-1})\Big] + \Delta_x \Big[x^{-\frac{n+1}{2}}  \Big] \theta_{L^{\sharp}}(x^{-1})\Bigg) \\ & + &
2 |\det(L)| \Bigg( \sum\limits_{A,B \subseteq \{1,2,\ldots,n\}} \langle {\rm grad} \;[x^{-\frac{n+1}{2}}]_A, {\rm grad}\;
[\theta_{L^{\sharp}}(x^{-1})]_B\rangle e_A e_B\Bigg), \nonumber
\end{eqnarray}
where a Clifford algebra valued expression $f$ is represented in its real components according to $f = \sum\limits_{A \subseteq \{1,2,\ldots,n\}} f_A e_A$ and where $\langle \cdot,\cdot\rangle$ stands for the standard scalar product on $\mathbb{R}^{n+1}$. The term $(\Delta_x \theta_L(x^{-1}))$ can be explicitly expressed by the monogenic theta series,
\begin{equation}\label{laplacetheta}
\Delta_x \Big[\theta_{L^{\sharp}}(x^{-1})\Big] = \frac{1}{|x|^4} \Bigg(  \Theta_{L^{\sharp}}^M(x^{-1}) - 4 x_0 [\frac{\partial \Theta_{L^{\sharp}}}{\partial x_0}](x^{-1})   + 4 \sum\limits_{i=1}^3 x_i [\frac{\partial \Theta_{L^{\sharp}}}{\partial x_i}](x^{-1})
\Bigg).
\end{equation}
\end{theorem}
\begin{proof}
To prove the formula we apply the Laplacian on both sides of the equation $\theta_L(x) = |\det(L)| x^{-\frac{n+1}{2}} \theta_{L^{\sharp}}(x^{-1})$. It is well known that the Laplacian applied to a product of two real-valued functions $f_A,g_B: \mathbb{R}^{n+1} \to \mathbb{R}$ satisfies the product rule
$$
\Delta(f_A \cdot g_B) = f_A(\Delta g_B) + (\Delta f_A) g_B + 2 \langle {\rm grad}\;f_A,{\rm grad} g_B\rangle.
$$
Now suppose that $f$ and $g$ are $\mathbb{R}_n$-valued functions, represented in the form $f(x) = \sum\limits_{A \subseteq \{1,\ldots,n\}} f_A(x) e_A$ and  $g(x) = \sum\limits_{B \subseteq \{1,\ldots,n\}} g_B(x) e_B$. Since $\Delta$ is a scalar-valued operator the previous formula gets the following form in the Clifford algebra valued case:
$$
\Delta(f \cdot g) = f (\Delta g) + (\Delta f) g + 2  \sum\limits_{A,B \subseteq \{1,2,\ldots,n\}}\langle {\rm grad}\;f_A,{\rm grad} g_B\rangle e_A e_B.
$$
Setting $f(x) = x^{-\frac{n+1}{2}}$ and $g(x) = \theta_{L^{\sharp}}(x^{-1})$ leads to the formula (\ref{monthetatrafo}). Applying next the formula
$$
\Delta_x [f(x^{-1})] = \frac{(\Delta_x f)(x^{-1})}{|x^4|} + \frac{1}{|x|^4} \Bigg(-4 \frac{\partial f}{\partial x_0}(x^{-1}) + 4 \sum\limits_{i=1}^3 \frac{\partial f}{\partial x_i}(x^{-1})      \Bigg)
$$
which can be verified by a direct computation,
leads  to (\ref{laplacetheta}) since $\Theta_L^M(x) = \Delta_x \Theta_L(x)$.
\end{proof}	
{\bf Statements and Declarations}. On behalf of all co-authors, the corresponding author states that there is no conflict of interest. There  are no data in support the findings of the study. Open access funding enabled and organized by Projekt DEAL.

\end{document}